\documentclass[11pt]{amsart}
\usepackage{}
\usepackage{amssymb}
\theoremstyle{plain}
\newtheorem{Thm}{Theorem}

\newtheorem{Pro}[Thm]{Proposition}

\newtheorem{Def}[Thm]{Definition}

\newcommand{\rad}{\operatorname{rad}}
\newcommand{\ep}{\epsilon}

\begin{document}

%begin Topmatter
\title[Proof of Hamilton's conjecture]
{A complete proof of Hamilton's conjecture}

\author{Li Ma}

\email{nuslma@gmail.com}

\thanks{The research is partially supported by the National Natural Science
Foundation of China 10631020 and SRFDP 20090002110019}

\begin{abstract}
In this paper, we give the full proof of a conjecture of R.Hamilton
that for $(M^3, g)$ being a complete Riemannian 3-manifold with
bounded curvature and with the Ricci pinching condition $Rc\geq \ep
R g$, where $R>0$ is the positive scalar curvature and $\ep>0$ is a
uniform constant, $M^3$ is compact. One of the key ingredients to
exclude the local collapse in singularities of the Ricci flow is the
use of pinching-decaying estimate. The other important part of our
argument is to role out the Type III singularity complete noncompact
Ricci flow with positive Ricci pinching condition. We get this goal
by obtaining an Ricci expander based on the monotonicity formula for
forward reduced volume.


\end{abstract}

 \maketitle

\section{Introduction}

Classical Myers theorem says that a complete Riemannian manifold
with Ricci curvature bounded below by a positive constant is
compact. Based on the study of Ricci flow, R.Hamilton conjectures
(see conjecture 3.39 in page 149 of the book \cite{Cho06}) that
for $(M^3, g)$ being a complete Riemannian 3-manifold with bounded
curvature and with the Ricci pinching condition $Rc\geq \ep R g$,
where $R>0$ is the positive scalar curvature and $\ep\in (0,1)$ is
a uniform constant, $M^3$ is compact. In this paper, we prove this
conjecture.

Note that this conjecture is a consequence of the following
result.
\begin{Thm}\label{thm:1}
Assume that $(M^3, g)$ is a 3-dimensional complete noncompact
Riemannian manifold with bounded sectional curvature. Suppose
$(M^3, g)$ satisfies the following Ricci pinching condition
\begin{equation}\label{eq:1}
R_{ij}\geq \ep Rg_{ij}, \quad on\ M^3
\end{equation}
for some uniform constant $\ep \in (0,1)$. Then $(M^3, g)$ is
flat. Here $R_{ij}$ is the Ricci tensor of $g=(g_{ij})$.
\end{Thm}

We prove Theorem \ref{thm:1} by using the Ricci flow introduced by
R. Hamilton. By definition, a family $(g(t))$ of Riemannian
metrics on $M^3$ is called a Ricci flow if $g(t)$ satisfies the
following Ricci flow equation
\begin{equation}\label{eq:2}
\partial_t g_{ij}(t)=-2R_{ij}(g(t)), \quad on \ M,
\end{equation}
with $g(0)=g$. In the following we shall assume that our Ricci flow
is not trivial, i.e., it is not flat. One of the key ingredients to
exclude the local collapse in singularities of the Ricci flow is the
use of pinching-decaying estimate in Proposition \ref{pro:key}
below. The other important part of our argument is to role out the
Type III singularity complete noncompact Ricci flow with positive
Ricci pinching condition. We get this goal by obtaining an Ricci
expander based on based on the monotonicity formula for forward
reduced volume (see \cite{CZ} and related \cite{M3}).

With the extra assumption that the sectional curvature is
non-negative (and in this case one has the injectivity radius
bound and Hamilton's Harnack differential inequality for Ricci
flow), the result has been proved in \cite{Che02}. The complex
version of Theorem \ref{thm:1} is a conjecture due to S.T.Yau (see
\cite{SY}). The above question was posed to R.Hamilton by W.X.Shi,
who was asked for this by S.T.Yau.

\section{Preliminary}
In this section, we firstly recall the concept of $k$-non-collapse
Ricci flow in the sense of G.Perelman \cite{P02}.

\begin{Def} Fix $k>0$. A complete solution $(M^n,g(t)$, $0\leq t<\infty$, to Ricci
flow is said $k$-non-collapse on all scales (or no local collapse
at all scales) if for all $(x_0,t_0)\in M\times [0,\infty)$ with
$|Rm|\leq r^{-2}$ on the parabolic ball
$$
P(x_0,t_0, r,-r^2):=\{(x,t); d(x,x_0;t)<r, t_0-r^2<t<t_0\},
$$
we have $Vol_{g(t_0)}(B(x_0,t))\geq kr^n$. Here $d(x,x_0;t)$ is
the distance function defined by the metric $g(t)$.
\end{Def}

We then cite Cheeger-Gromov-Taylor's result, which is in use of
bounding the injectivity radius for Ricci flow.

\begin{Pro}\label{pro:1}(Cheeger-Gromov-Taylor \cite{Che82})
Assume that $(M^n, g)$ is a complete Riemannian manifold of
dimension n and $p\in M$. $\forall \ep>0$, $\exists C_0 \sim  n,
\ep$ with the following property. Suppose that
$$
|Rm(x)|\leq r^{-2}, \quad for\ all\ x\in B(p, r)
$$
and $Vol(B(p, r))\geq \ep r^n$. Then $inj_g(M, p)\geq C_0>0$.
\end{Pro}

One important tool in our study is Hamilton's compactness Theorem
for Ricci flows (Hamilton, 1995 \cite{Ham95}; see also Theorem
6.35 in \cite{Cho06}), which will not be stated explicitly in this
work.

We now prove a key pinching estimate, which will play a important
role to rule out the singularities of the Ricci flow.

\begin{Pro}\label{pro:key}
Assume that $(M^3, g(t))$, $\omega\leq t\leq T< +\infty$, is a
Ricci flow on a 3-dimensional complete noncompact Riemannian
manifold with bounded sectional curvature. Suppose $(M^3, g(0))$
satisfies the following Ricci pinching condition (\ref{eq:1}) for
some uniform constant $\ep \in (0,1)$. Then (\ref{eq:1}) is
preserved and moreover, we have the following decay estimate
\begin{equation}\label{estimate:key}
|Rc(g(t))-\frac{1}{3}R(t)g(t)|^2\leq
(\frac{3}{2(t-\omega)})^{\ep^2}R(t)^{2-\sigma}
\end{equation}
for $t\in [\omega,T)$.
\end{Pro}

\begin{proof}
Without loss of generality, we may assume that $\omega=0$.

 Using
Hamilton's derivation of the evolution equation for the Ricci
tensor $Rc(g(t))$ and using Shi's maximum principle on complete
noncompact manifolds, we know that the condition (\ref{eq:1}) is
preserved by the Ricci flow and furthermore, there exists $C>0$
and $\delta>0$ depending only on $g$ such that
\begin{equation}\label{eq:4}
\frac{|Rc(t)-\frac{1}{3}R(t)g(t)|}{R(t)}\leq CR(t)^{-\delta},\quad
on\ M\times (0,T).
\end{equation}

 Let $\sigma=\ep^2$ and let
$$
f(t)=f_{\sigma}
(t)=R(t)^{\sigma-2}|Rc(g(t))-\frac{1}{3}R(t)g(t)|^2.
$$
Following the computation of Hamilton (Lemma 10.5 in
\cite{Ham82}), we know that
\begin{eqnarray*}
f_t & \leq & \Delta f+\frac{2(1-\sigma)}{R(t)}<\nabla R(t), \nabla f>\\
&   & +2R(t)^{\sigma-3}[\sigma
|Rc(g(t))|^2|Rc(g(t))-\frac{1}{3}R(t)g(t)|^2-2P],
\end{eqnarray*}
where $P$ satisfies (see Lemma 10.7 in \cite{Ham82})
$$
P\geq \sigma |Rc(g(t))|^2|Rc(g(t))-\frac{1}{3}R(t)g(t)|^2
$$
since $Rc(g(t))\geq \ep R(t)g(t)$. It is clear that
$$
2P-\sigma|Rc(g(t))|^2|Rc(g(t))-\frac{1}{3}R(t)g(t)|^2\geq
\frac{1}{3}\sigma R(t)^{3-\sigma}f^{1+\frac{1}{\sigma}},
$$
and then
$$
f_t\leq \Delta f+\frac{2(1-\sigma)}{R(t)}<\nabla R(t), \nabla
f>-\frac{2}{3}\sigma f^{1+\frac{1}{\sigma}}.
$$
Using Shi's maximum principle, we conclude that
\begin{equation}\label{eq:5}
f_\sigma (t)\leq (\frac{3}{2t})^\sigma.
\end{equation}
This complete the proof.
\end{proof}

We shall use this estimate for the local limit Ricci flow at $t=0$
for type I and Type IIa Ricci flows and let $\omega\to -\infty$ to
get an Einstein metric with positive curvature, which will gives
us a contradiction under various geometric considerations. We
remark that similar estimate for compact ancient solution has been
obtained in \cite{B}.

In the following result we can exclude the local collapse in Type
I or II singularity Ricci flow with positive Ricci pinching in
dimension three.

\begin{Pro}\label{pro:2}
Assume that $(M^3, g(t))$, $0\leq t<T<\leq\infty$, is a Type I or
Type II Ricci flow on a 3-dimensional complete noncompact
Riemannian manifold with bounded sectional curvature. Suppose
$(M^3, g(0))$ satisfies the following Ricci pinching condition
(\ref{eq:1}) for some uniform constant $\ep \in (0,1)$. Then there
is no local collapse in the flow near the maximal time $T$.
\end{Pro}
\begin{proof}
  We shall divide the proof into two cases.

\textbf{ Case One}:
 In type I singularity model of Ricci flow with
 positive Ricci pinching (or type IIa singularity), if it is local
 collapse at the blow up point, then it is not bump-like point
 in the sense of Hamilton's paper\cite{Ham95} (see also the paper
 \cite{Chow1} for more detail).
  Otherwise, we have the injectivity radius bound.
  Hence it is a split-like point of the sequence of almost non-negative sectional curvature
  (in short ANSC) in the sense of the paper \cite{Chow2}).
   Using the Fukaya-Glickenstein lifting trick \cite{G03} to the unit ball of the tangent space,
    we get a local limit Ricci flow
   with at one zero eigenvalue of the Ricci tensor. By the Ricci pinching condition
   we know it is flat, a contradiction with the normalization condition $|Rc(g)|=1$ at the blow up point.

\textbf{Case Two}: In the type IIb singularity Ricci flow, if the
blow up point is local collapse, again we meet the split-like
point (see Cor. 9. in the paper \cite{Chow2}). Using the same
Fukaya-Glickenstein lifting trick as above, we find the
contradiction again.

\end{proof}

We remark that we can also prove the above results by the pinching
estimate (\ref{estimate:key}) and the method in \cite{M2}.

\section{Proof of Theorem \ref{thm:1}}

We may assume that we have a maximal complete non-compact Ricci
flow with positive Ricci pinching condition (\ref{eq:1}) on the
manifold $M^3$. By lifting to universal covering, we may assume
that $M$ is simply-connected.

 Along the Ricci flow, we have
\begin{equation}\label{eq:3}
\partial_t R(t)=\Delta R+2|Rc|^2,
\end{equation}
where $Rc=Rc(g(t))$ is the Ricci tensor of $g(t)$, $R(g(t))$ is
the scalar curvature of $g(t)$, and $\Delta=\Delta_{g(t)}$ is the
Laplacian-Beltrami operator of the metric $g(t)$. Assume that $(M,
g)$ is non-flat. By (\ref{eq:3}), we know that $R(t)>0$.  By the
result of W.X. Shi \cite{Shi89a} \cite{Shi89b}, we know that there
is a (maximal) positive $T>0$ such that the Ricci flow $g(t)$
exists with $g(0)=g$ and if $T<+\infty$, then
$$
\sup_M |Rm(g(t))|\to +\infty, \quad on \ M^3,
$$
as $t\to T$, where $Rm(g(t))$ is the Riemannian curvature tensor of $g(t)$.

Claim: $T=+\infty$. Assume not. Then choose $(x_k, t_k)\in
M\times(0, T)$, $t_k\to T$, and $\gamma_k \nearrow 1$ such that
$$
R_k=R(x_k, t_k)\geq \gamma_k \sup_{M^3\times (0, t_k)} R(x, t)\geq \gamma_k \sup_{M^3\times (0, t_k)} |R_m(g(t))|\to +\infty,
$$
as $k\to +\infty$.
Define
$$
g_k(t)=R_kg(t_k+R_k ^{-1}t), \quad t\in (-t_kR_k, 0),
$$
and consider the sequence of solutions $(M, g_k(t), x_k)$. Let
$r_k=R_k^{-\frac{1}{2}}\leq 1$ for $k$ large. We have
$$
|Rm(\cdot , t_k)|\leq 2r_k^{-2}, in \ B_{g(t_k)}(x_k, r_k).
$$
Then by Proposition \ref{pro:2}, we know that
 $inj_{g_k(0)}(x_k)\geq C_0$. Hence, we can
apply Hamilton's compactness theorem \cite{Ham95} to conclude a
subsequence $(M, g_k(t), x_k)$, which converges to $(M_{\infty},
g_{\infty}(t), x_\infty)$, a complete solution with $t\in
(-\infty, 0]$.

Note that
$$
R_{g_\infty}(x_\infty, 0)=\lim_{k\to \infty} R_{g_k}(x_k, 0)=\lim_{k\to \infty}1=1.
$$
Hence $R_{g_\infty}(x, 0)>0$
in a neighborhood of $x_\infty$. By the estimate (\ref{eq:4}) we know that
$$
\frac{|Rc(g_k)-\frac{1}{3}R(g_k)g_k|}{R(g_k)}(t)\leq CR_k ^{-\delta}R(g_k)^{-\delta}\to 0,
$$
as $k \to \infty$. Then we have
$$
Rc(g_\infty)=\frac{1}{3}R(g_\infty)g_\infty
$$
on the subset of $M_\infty$ where $R(g_\infty)>0$. Using the contracted
second Bianchi identity, we know that $R(g_\infty)=constant$ in any connected
neighborhood of $x_\infty$ in $M_\infty$. Hence
$$
R(g_\infty)=1,\ on\ all\ of\ M_\infty
$$
and $M_\infty$ is compact, which is a contradiction with M being
noncompact (since compact $M_\infty$ can not be geometric limit of
a sequence of complete non-compact Riemannian manifolds in the
sense of Cheeger-Gromov). One may see \cite{MC} for more
compactness results.

From the argument above, we can see that $T=\infty$ and $R(g(t))$
is uniformly bounded for $t\in [0, \infty)$.

Assume that our Ricci flow has no local collapse. We now
re-normalize $g(t)$ such that we can use Hamilton's compactness
theorem. Choose $(\bar{x}_k, \bar{t}_k)\in M\times(0, +\infty)$,
$\bar{t}_k\to +\infty$ and some $\delta>0$ small such that
\begin{eqnarray*}
+\infty>C_2\geq \bar{R}_k=R(\bar{x}_k, \bar{t}_k) & \geq & \delta\sup_{M^3\times (0, \bar{t}_k)}|R(x, t)|\\
& \geq & \delta\sup_{M^3\times (0, \bar{t}_k)}|Rm(g(t))|\geq
C_1>0,
\end{eqnarray*}
as $k\to \infty$. Define
$$
\bar{g}_k(t)=\bar{R}_kg(\cdot, \bar{t}_k+\bar{R}_k^{-1}t),\quad  t
\in (-\bar{R}_k\bar{t}_k, +\infty),
$$
and consider $(M, \bar{g}_k(t), \bar{x}_k)$. Once again we have
$$
inj_{\bar{g}_k(0)}(\bar{x}_k)\geq C_0(\delta)>0,
$$
(here we have used Proposition \ref{pro:1} and the fact that the
uniform curvature bound of Ricci flows) and by the compactness
theorem, we get a complete solution $(M, \bar{g}(t), x_\infty)$,
$(t\in (-\infty, +\infty))$, which is the geometric limit of $(M,
\bar{g}_k(t), x_k)$. Using the estimate (\ref{eq:5}) to
$\bar{g}(t)$, we know that
$$
|Rc(\bar{g}(t))-\frac{1}{3}R(\bar{g}(t))\bar{g}(t)|^2=0,
$$
which implies that $(M^3,\bar{g}(t))$ is a complete Riemannian
manifold with positive constant curvature. This again implies that
M is compact. A contradiction.

 Assume that we are
giving a type III singularity Ricci flow with
 positive Ricci pinching condition and with local collapse.
 Take any sequence $t_n\to \infty$ and pick the blow-up point $x_n$
 for the metric $g(t_n)$ and define the blow-up sequence $(M,x_n,g_n(t))$
 with $g_n(t)=R(x_n,t_n)g(R(x_n,t_n)^{-1}t+t_n)$. Assume that we have the local
 collapse for each $g_n$ at $x_n$. Take $0<\epsilon_0<<1$. Define
  $$rad(x_n)=\inf\{r>0, Vol_{g_n}(x_n,r)\leq \epsilon_0 r^3\}$$
Then we have $0<rad(x_n)<\infty$ and $rad(x_n)\to 0$. In fact, on one hand, note that we have
$$\lim_{r\to\infty}\frac{Vol_{g_n}(x_n,r)}{r^3}=0.$$
On the other hand, using $Ric>0$ and the Bishop-Gromov volume comparison we know that for $r\leq rad(x_n)$,
$$
\frac{Vol_{g_n}(x_n, r)}{r^3}\geq \epsilon_0.
$$
Then using the Cheeger-Gromov-Taylor theorem we know that $inj_{g_n(0)}(x_n)\geq C rad(x_n)$ for some
uniform constant $C>0$. Using the local collapse condition we then know
that $rad(x_n)\to 0$. Otherwise $g_n$ is not local collapse at
$x_n$.

We now consider the new blow-up sequence
$$
\hat{g}_n(t)=\frac{R(x_n,t_n)}{rad(x_n)^2}g_n(\frac{\rad(x_n)^2}{R(x_n,t_n)}t+t_n)
$$
at $x_n\in M$. Note that the curvature of $\hat{g}_n(t)$ is
uniformly bounded by $Crad(x_n)^2\to 0$. Using the Hamilton's
Cheeger-Gromov compactness for Ricci flow, we get a limit Ricci
flow of constant zero curvature, which is a contradiction to the
fact that the volume of the unit ball $B(x_n,1)$ in the metric
$\hat{g}_n(0)$, that is the rescaled ball
$$
\frac{1}{rad(x_n)}B_{g_n}(x_n,rad(x_n))
$$
in the metric $g_n(0)$ converges to $\epsilon_0$ and according to
the volume convergence lemma of  T.Colding \cite{Colding}, it
should be $vol(B(0,1))$ in the limit metric, which is the standard
Euclidean metric.

Assume that we are giving a type III singularity Ricci flow with
positive Ricci pinching condition and with no local collapse at
the scale $r_0<\infty$. Choose $\nu_0\in (0,1/8)$ and $x_n\in M$
and $r_n\in (0,r_0]$ such that
$$r_n=\sup\{r; , r^{-3}Vol(x_n,r)\geq
\nu_0, \ \ for \ \ \rho\in(0,r)\}.$$ Consider the pointed
manifolds
$$
(M, r_n^{-2}g(t_n+r_n^2t),x_n),
$$
with curvature control $$ |Rc(r_n^{-2}g(t_n+r_n^2t)|\leq
Cr_n/(t_n+r_n^2t)\to 0
$$
 and whose limit is the flat space $(M_{\infty},g_{\infty},x_{\infty})$. Using Colding's volume
 convergence lemma \cite{Colding} at $t=0$ we get that
 $$\nu_0=\lim Vol_{r_n^{-2}g(t_n)}(B(x_n,1))=Vol_{g_{\infty}}(B(x_{\infty},1))
 $$
 which gives us a contradiction again.

 Using the similar argument for the sequence of metrics
 $g(t_n)$ at any point sequence $x_n\in M$ there are a
 uniform constant $k_0>0$ and subsequence $t_n>0$ such that
\begin{equation}\label{+vol}
Vol_{\bar{g}_n}B(x_n,1)=(\frac{Vol_{g(t_n)}B(x_n, \sqrt{t_n})}{(\sqrt{t_n})^3}\geq k_0.
\end{equation}
It is this fact used to study the global behavior of the heat kernel
associated with the Type III Ricci flow (\cite{M2}). In fact, if
this property is not true, we consider the rescaled metrics
 $\bar{g}_n:=t_n^{-1}g(t_n)$ and then we have
$$
Vol_{\bar{g}_n}B(x_n,1)\to 0.
$$
Take $0<\epsilon_0<<1$. Define
  $$rad(x_n)=\inf\{r\in (0,1), Vol_{\bar{g}_n}(x_n,r)\leq \epsilon_0 r^3\}$$
If $rad(x_n)\to 0$, arguing for the new blow-up sequence as above,
we get a contradiction to Colding's result. Hence we have some
positive constant $\delta_0>0$ such that $rad(x_n)>\delta_0$
uniformly which gives (\ref{+vol}).

Hence it remains to exclude the possibility of a type III
singularity complete non-compact Ricci flow with positive Ricci
pinching condition and with no local collapse at all scales. This
can be actually proved in the similar way as we do in \cite{M2} by
using instead based on the monotonicity formula for forward reduced
volume ( similar to the forward reduced volume introduced by
Feldman-Ilmenan-Ni \cite{FIN}). For completeness, we give the some
detail in the appendix.

 This completes the
proof of Theorem \ref{thm:1}.

\section{Appendix}

Assume that $(M^3,g(t))$ is a type III singularity complete
non-compact Ricci flow with positive Ricci pinching condition and
with no local collapse at all scales. In this case we just choose
some time sequence $t_k\to\infty$, a point sequence $x_k\in M$,
and scales $Q_k\to 0$ with the metric
$$
g_k(s)=Q_kg(\cdot, t_k+sQ_k^{-1}), s>0,
$$
such that $|Rm(g_k)(x_k,0)|\geq c>0$ for some uniform constant $c$
and $(M,g_k(s),x_k)$ converges to
$(M_{\infty},g_{\infty},x_{\infty})$ with $R(g_{\infty}(s))\leq
\frac{A}{A+s}$ for some positive constant $A>0$. We remark that we
can take any time sequence $t_k$ and any point sequence $x_k\in
M$,say, $x_k = x$ being a fixed point, such that the rescaled metric
$g_k$ converges $g_\infty$ in the sense of Cheeger-Gromov.

We now review the monotonicity formula of the weighted forward
reduced volume (introduced by us, see \cite{CZ} and \cite{FIN}). For
our Ricci flow $(M^n,g(t))$, $(n\geq 3$, and a path
$\gamma(\eta),\eta)$ connecting $(x,0)$ and $(y,t)$. We define the
forward $\mathbb{L}_+$-length of $\gamma(\eta)$ by
$$
\mathbb{L}_+(\gamma)=\int_0^t\sqrt{\eta}(R(\gamma(\eta))+|\gamma'(\eta)|^2)d\eta.
$$
Then we can define the $\mathbb{L}_+$-length $L_+(y,t)$ from $(y,t)$
to $(x,0)$ and the forward $l_+$-length
$$
l_+(y,t)=\frac{L_+(y,t)}{2\sqrt{t}}.
$$
As in \cite{P02} and \cite{FIN} we can define the forward
$\mathbb{L}_+$-exponential map $\mathbb{L}_+(\cdot,t):T_xM\to M$ at
$V\in T_xM$ and at time $t$ by $\mathbb{L}_+exp(V,t)=\gamma_V(t)$,
which is the maximal $\mathbb{L}_+$-geodesic emanating from $x$ and
$\lim_{\eta\to 0}\sqrt{\eta}\gamma'_V(\eta)=V$. Denote by
$|V|^2=g(x,0)(V,V)$ the euclidean length of the vector $V$. As in
\cite{P02} (see Proposition 6.78 in \cite{MT}), we can show that for
$y=\gamma_V(t)$, as $\tau\to 0$, the quantity $$
\tau^{-n/2}e^{l_+(y,\tau)}dv_{g(y,\tau)}
$$
is non-increasing and
$$
\tau^{-n/2}e^{l_+(y,\tau)}dv_{g(y,\tau)}\to 2^ne^{|V|_{g(0)}^2}.
$$
Then we can define the weighted forward volume by
\begin{equation}\label{e-volume}
\mathbb{V}_+(\tau)
=\int_{\Omega(\tau)}\tau^{-n/2}e^{l_+(y,\tau)}e^{-2|\mathbb{L}_+exp^{-1}(y,\tau)|^2}\mathbb{J}_+(V,\tau)dV,
\end{equation}
where $\Omega(\tau)$ is the injectivity domain of $\mathbb{L}_+exp$
at $(x,0)$ and $\mathbb{J}_+(V,\tau)$ is the Jacobian determinant of
$\mathbb{L}_+exp$ at the point $V\in \Omega(\tau)$. We remark that
we may change the base point $(x,0)$ (and $x_k=x$ in the blow-up
sequence $g_k$) to any point $(z,0)$ without changing the weighted
volume.

The important feature of the weighted forward reduced volume
$\mathbb{V}_+(t)$ is that it is a metric scale-invariant and it is
monotone non-increasing in $t$ and if for some $t_1<t_2$,
$\mathbb{V}_+(t_1)=\mathbb{V}_+(t_2)$, then the Ricci flow is a
gradient expanding soliton on $[0,\infty)$. For a detailed proof,
one may see \cite{CZ}.

We now apply the monotonicity formula above to the Ricci flow
$g_k(s)$. Let $\tau_k=Q_k^{-1}$ and let
$V_{+k}(s)=\mathbb{V}_+(s\tau_k)$. Fix $s=s_0\in [1,2]$. Then by the
monotonicity of weighted forward volume, we can choose the sequence
$t_k\to\infty$ such that
$$
\mathbb{V}_+(s_0\tau_k)\geq \mathbb{V}_+((s_0+1)\tau_k)\geq
V_{+k}(s_0)=\mathbb{V}_+(s_0\tau_{k+1}).
$$
Then we have
$$\lim_{k\to\infty}(V_{+k}(s_0)-V_{+k}(s_0+1))=0$$ which implies that
$$
\mathbb{V}_+(s_0,g_\infty)=\mathbb{V}_+(s_0+1,g_\infty)
$$
and for $s\in [s_0,s_0+1]$,
$$
Ric_{\infty}-\frac{1}{2\sqrt{s}}Hess_{L_{+\infty}}+\frac{1}{2s}g_{\infty}=0.
$$
By our choice of $Q_k$ in the blow-up sequence, we have that
$g_{\infty}$ is not flat. So we have $Rc(g_{\infty})>0$ everywhere
by the pinching condition.
 Using the Ricci pinching condition and the results in \cite{MD},
 we know that $(M_{\infty},g_{\infty})$ is an asymptotic flat
 manifold with one end. Then invoking
  the results in \cite{MM} and \cite{MD},
 we know that $(M_{\infty},g_{\infty})$ is again isometric to $R^3$,
which is impossible too. Hence $M$ is compact.

\emph{Acknowledgement}: The author would like to thank Baiyu Liu for
typing the paper, Liang Cheng for his interest, Profs. S.T.Yau and
S.Brendle for helpful comments by email. The author would also like
to thank IHES, France for hospitality and the K.C.Wong foundation
for support in 2010. The author would like to thank the referees
very much for suggestions.

\end{document}